\documentclass[a4paper,11pt,leqno]{article}

\usepackage[latin1]{inputenc}
\usepackage[T1]{fontenc} 
\usepackage[english]{babel}
\usepackage{verbatim}
\usepackage{calligra}
\usepackage{enumerate}
\usepackage{dsfont}
\usepackage{amsfonts}
\usepackage{amsmath}
\usepackage{amsthm}
\usepackage{amssymb}
\usepackage{mathtools}
\usepackage{mathrsfs}
\usepackage{color}
\usepackage{graphicx}
\usepackage{bm}
  \usepackage[all]{xy}
\usepackage{hyperref}
\hypersetup{
    colorlinks=true,
    linkcolor=black,
    filecolor=black,      
    urlcolor=black,
}

\usepackage{tikz}
\usepackage[retainorgcmds]{IEEEtrantools}

\usepackage{graphicx}		
\usepackage[hypcap=false]{caption}

\DeclareMathOperator*{\tend}{\longrightarrow}

\DeclareMathOperator*{\wtend}{\rightharpoonup}
\DeclareMathOperator*{\limss}{\overline{\rm lim}}

\theoremstyle{definition}
\newtheorem{defi}{Definition}
\newtheorem{ex}[defi]{Example}
\newtheorem{rmk}[defi]{Remark}

\theoremstyle{plain}
\newtheorem{thm}[defi]{Theorem}
\newtheorem{prop}[defi]{Proposition}
\newtheorem{cor}[defi]{Corollary}
\newtheorem{lemma}[defi]{Lemma}

\newcommand{\mc}{\mathcal}

\newcommand{\what}{\widehat}

\newcommand{\vphi}{\varphi}

\newcommand{\R}{\mathbb{R}}
\newcommand{\Q}{\mathbb{Q}}

\newcommand{\N}{\mathbb{N}}
\newcommand{\Z}{\mathbb{Z}}

\newcommand{\dx}{ \, {\rm d} x}

\newcommand{\B}{\dot{B}^s_{p, r}}
\newcommand{\BB}{\dot{B}^s_{p, \infty}}
\newcommand{\CC}{C^s_p}

\allowdisplaybreaks

\begin{document}

\newcommand{\dimitri}[1]{\textcolor{red}{[***DC: #1 ***]}}
\newcommand{\fra}[1]{\textcolor{blue}{[***FF: #1 ***]}}

\title{\textsc{\Large{\textbf{Remarks on Chemin's space of homogeneous distributions}}}}

\author{\normalsize\textsl{Dimitri Cobb}\vspace{.5cm} \\
\footnotesize{\textsc{Universit\'e de Lyon, Universit\'e Claude Bernard Lyon 1}} \\
{\footnotesize \it Institut Camille Jordan -- UMR 5208} \vspace{.1cm} \\
\footnotesize{\ttfamily{cobb@math.univ-lyon1.fr}}
}

\vspace{.2cm}
\date\today

\maketitle

\begin{abstract}
This article focuses on Chemin's space $\mc S'_h$ of homogeneous distributions, which was introduced to serve as a basis for realizations of subcritical homogeneous Besov spaces. We will discuss how this construction fails in multiple ways for supercritical spaces. In particular, we study its intersection $X_h := \mc S'_h \cap X$ with various Banach spaces $X$, namely supercritical homogeneous Besov spaces and the Lebesgue space $L^\infty$. For each $X$, we find out if the intersection $X_h$ is dense in $X$. If it is not, then we study its closure $C = {\rm clos}(X_h)$ and prove that the quotient $X/C$ is not separable and that $C$ is not complemented in $X$.
\end{abstract}
\paragraph*{2020 Mathematics Subject Classification:}{\small 42B35 
(primary);
46E30, 
46E35, 
42B37 
(secondary).}


\section{Introduction}

The purpose of this short article is to study the role of Chemin's space $\mc S'_h$ of homogeneous distributions in the structure of supercritical\footnote{In this article, the term \textsl{supercritical} will refer to the hard case, where homogeneous Besov spaces \textsl{are not} realized as subspaces of $\mc S'_h$. (Strictly) supercritical exponents for $\B$ are $s > d/p$, or $s=d/p$ and $r > 1$.} homogeneous Besov spaces. The space $\mc S'_h$ was introduced by J.-Y. Chemin in the mid 90s as being the set of tempered distributions $f \in \mc S'$ that satisfy the following low frequency condition:
\begin{equation*}
\chi(\lambda \xi) \what{f}(\xi) \tend_{\lambda \rightarrow + \infty} 0 \qquad \text{in } \mc S',
\end{equation*}
where $\chi \in \mc D$ is a compactly supported cut-off function having value $\chi \equiv 1$ in a neighborhood of $\xi = 0$. This must be seen as a (very weak) description of the behavior of $f$ at infinity $|x| \rightarrow + \infty$.

\medskip

Homogeneous Besov spaces are defined as subspaces of the quotient space of tempered distributions modulo polynomials $\mc S' / \, \R[X]$ endowed with an appropriate norm expressed in terms of the homogeneous Littlewood-Paley decomposition (see Definition \ref{d:c2SpH} below). In other words,
\begin{equation*}
\B = \Big\{ f \in \mc S' / \, \R[X], \quad \| f \|_{\B} < + \infty \Big\}.
\end{equation*}
On the other hand, the space $\mc S'_h$ contains no polynomial function, so the natural embedding $\mc S'_h \hookrightarrow \mc S' / \, \R[X]$, makes it possible to consider $\mc S'_h \cap \B$ as a subspace of $\B$. In this sense, it is well known that $\B \subset \mc S'_h$ as long as the space $\B$ is subcritical, that is as long as the indices $(s, p, r) \in \R \times [1, + \infty]^2$ satisfy
\begin{equation}\label{ieq:superCrit}
s < \frac{d}{p} \qquad \text{or} \qquad s = \frac{d}{p} \text{ and } r = 1.
\end{equation}
In other words, under condition \eqref{ieq:superCrit}, all $f \in \B \subset \mc S' / \, \R[X]$ are the image of an element of $\mc S'_h$ by the embedding $\mc S'_h \hookrightarrow \mc S' / \, \R[X]$. In fact, this property of subcritical homogeneous Besov spaces is the reason for which the space $\mc S'_h$ was introduced: it was to serve as a basis for realizations of $\B$ \cite{Chemin-cours}, \cite{Danchin-cours}, \cite{BCD}. 

\medskip

However, although it is common knowledge that $\mc S'_h \cap \B \subsetneq \B$ if the supercriticality condition \eqref{ieq:superCrit} does not hold,\footnote{In the sense that there are (classes of) distributions $f \in \B$ such that $f \notin \mc S'_h \oplus \R[X]$.} little has been said about this strict inclusion. Our goal, in this short article, is to find out how far this inclusion actually is from being an equality.

\medskip

More precisely, we will prove three properties, which hold as soon as \eqref{ieq:superCrit} does not:
\begin{enumerate}[(i)]
\item the strict subspace $\mc S'_h \cap \B$ is dense in $\B$ if and only if $r < + \infty$;
\item when $r = + \infty$, note $\CC := {\rm clos}(\mc S'_h \cap \BB)$ the closure of the intersection in the $\BB$ toplogy; then the quotient space $\BB / \CC$ is not separable;
\item if $p < + \infty$, the space $\CC$ is not complemented in $\BB$. In other words, there is no decomposition $\BB = \CC \oplus G$ with continuous projections.
\end{enumerate}

Let us comment a bit on these statements. While for $r < + \infty$ there seems not to be much ``in between'' $\mc S'_h \cap \B$ and $\B$, it is not so for $\BB$. The fact that $\BB / \CC$ is not separable means that the inclusion $\mc S'_h \cap \BB \subsetneq \BB$ is indeed very far from being an equality.

The third result regarding the non-complementation of $\CC$ must be seen in the same way. As we will see, the proof relies on the construction of an uncountable family of relatively ``independent'' subspaces of $\BB / \CC$, and is in fact a direct prolongation of the proof we give of (ii).

\medskip

The study of complementation in Banach spaces is by no means new in the landscape of functional analysis, and has been marked by a number of very deep results (we refer to \cite{CQ} and the references therein for an enlightening introduction to the topic). Amongst these is the Phillips-Sobczyk theorem \cite{P}, \cite{S} which states that the space $c_0(\N)$ of sequences converging to zero is not complemented in the space $\ell^\infty(\N)$ of bounded sequences: there is no bounded projection $P : \ell^\infty(\N) \tend c_0(\N)$ (see also \cite{AK} Section 2.5 pp. 44-48).

Our result regarding the non-complementation of $\CC$ in $\BB$ is an adaptation of a proof of the Phillips-Sobczyk theorem by Whitley \cite{W}, which is based on a countability argument. As we will see, a well chosen embedding $J : \ell^\infty (\N) \tend \BB$ will allow the main ingredients of Whitley's proof to find their counterpart in the framework of $\BB$.

\medskip

We do certainly not presume to bring any meaningful contribution to the topic of complementation in Banach spaces: our goal is merely to use this theory to illustrate the role of $\mc S'_h$ in the structure of homogeneous Besov spaces.

\medskip

We point out that, besides serving as a basis for realizations of subcritical homogeneous Besov spaces, the space $\mc S'_h$ is also involved in the theory of bounded solutions in incompressible hydrodynamics \cite{Cobb}, where uniqueness for the initial value problem depends on low frequency characterizations of the solutions.

\medskip

Let us give a short overview of our paper. In Section \ref{s:c2EQSpH}, we start by giving general definitions on the topic of homogeneous Besov spaces, as well as discussing general properties and alternative definitions of the space $\mc S'_h$. We then turn to the investigation of the three points stated above. More precisely, in Section \ref{s:c1ClosB}, we prove (i) in the form of Theorem \ref{t:c2closB}. Next, in Section \ref{s:c1SepB}, we focus on the separability issue (ii), which is contained in Theorem \ref{t:c2nonSepB}. In Section \ref{s:c1ComplB}, we move on to (iii), which we prove in Theorem \ref{t:c2ComplB}. 

Finally, we end the article with a study of a critical case: that of the space $L^\infty$, which is intermediate between two Besov spaces that behave very differently in light of points (i), (ii) and (iii) above. Consider the chain
\begin{equation*}
\xymatrix{
\dot{B}^{0}_{\infty, 1} \ar[r]^\subset & L^\infty \ar@{->>}[r] & L^\infty / \, \R \ar[r]^\subset & \dot{B}^{0}_{\infty, \infty},
}
\end{equation*}
whose leftmost element $\dot{B}^0_{\infty, 1}$ is subcritical and is contained in the image of $\mc S'_h$ modulo polynomials, and whose rightmost element $\dot{B}^0_{\infty, \infty}$ is superctitical and behaves according to (i), (ii) and (iii) with $p = + \infty$ and $s=0$. In Section \ref{s:c2Linfty}, we try to replace $\B$ and $\BB$ by $L^\infty$ in points (i), (ii) and (iii).

\subsection*{Notations}

In this paragraph, we present some of the notations we will be using throughout the paper.

\begin{itemize}
\item Unless otherwise mentioned, all function spaces will be set on $\R^d$. In that case, we will omit the reference to $\R^d$ in the notation. For instance, $L^p = L^p(\R^d)$ is the set of complex valued functions on $\R^d$.
\item For $p \in [1, + \infty]$, we note $\ell^p(\N)$ and $\ell^p(\Z)$ the usual sequence spaces. The notation $\ell^p$ alone stands for $\ell^p(\N)$. The space $c_0$ is the space of sequences on $\N$ which converge to zero.
\item We note $\mc S$ the space of Schwartz functions and $\mc S'$ the space of tempered distributions, and $\mc D$ is the space of $C^\infty$ compactly supported functions.
\item If $X$ is a Frechet space whose (topological) dual is $X'$, we note $\langle \, . \, , \, . \, \rangle_{X' \times X}$ the associated duality bracket. The bracket $\langle \, . \, , \, . \, \rangle$ without mention to any space is associated to the $\mc S' \times \mc S$ duality.
\item In all that follows, $C$ is a generic constant that may change value from one line to another. When needed, we will specify the useful dependencies of the constant by the notation $C(\, . \,)$. If $a$ and $b$ are two nonnegative quantities, we note $a \lesssim b$\index{$a \lesssim b$} for $A \leq C b$ and $a \approx b$\index{$a \approx b$} for
\begin{equation*}
\frac{1}{C} a \leq b \leq C a.
\end{equation*}
\end{itemize}

\subsection*{Acknowledgements}

I am grateful to my advisor Francesco Fanelli for his help and support while writing this article, even as it it represented a slight detour from our research. I also am deeply indebted to Denis Choimet for commnuicating his passion for analysis to his students both in class and in book \cite{CQ}.

This work has been partially supported by the project
CRISIS (ANR-20-CE40-0020-01), operated by the French National Research Agency (ANR).

\section{Preliminary Definitions and Results}

This Section is devoted to stating a small number of results and definitions which we will need in the following pages. In particular, we will give a few basic properties of the space $\mc S'_h$ as well as the definition of homogeneous Besov spaces.

\subsection{Quotients of Banach Spaces}

In the sequel, we will need to work with quotients of Banach spaces, such as $\ell^\infty / c_0$. Under the right conditions, these have a norm structure which makes them complete. We refer to paragraph 1.39 in \cite{Douglas} for a proof of the Proposition below.

\begin{prop}
Let $X$ be a Banach space and $C \subset X$ be a closed subspace. We equip the quotient $X/C$ with the following norm:
\begin{equation*}
\forall x \in X/C, \qquad \| x \|_{X/C} = \inf_{h \in C} \| x - h \|_X.
\end{equation*}
Then the norm topology of $X/C$ is equal to the quotient topology and the natural projetion $\pi : X \tend X/C$ is bounded.
\end{prop}

\subsection{Homogeneous Littlewood-Paley Theory}\label{ss:c2LPdecomp}

In this Subsection, we present the basic elements of Littlewood-Paley theory which we will need in order to define the $\mc S'_h$ and Besov spaces.

\medskip

First of all, we introduce the Littlewood-Paley decomposition based on a dyadic partition of unity with respect to the Fourier variable. We fix a smooth radial function $\chi$ supported in the ball $B(0,2)$, equal to $1$ in a neighborhood of $B(0,1)$
and such that $r\mapsto \chi(r\,e)$ is nonincreasing over $\R_+$ for all unitary vectors $e\in\R^d$. Set $\varphi\left(\xi\right)=\chi\left(\xi\right)-\chi\left(2\xi\right)$ and $\vphi_m(\xi):=\vphi(2^{-m}\xi)$ for all $m \in \Z$. The homogeneous dyadic blocks $(\dot{\Delta}_m)_{m \in \mathbb{Z}}$ are defined by Fourier multiplication, namely\footnote{For any function $\sigma(\xi)$, we define the Fourier multiplier $\sigma(D)$ of symbol $\sigma(\xi)$ as being the operator given by the map $f \longmapsto \mc F^{-1}[\sigma(\xi) \what{f}(\xi)]$.}
\begin{equation*}
\forall m \in \mathbb{Z}, \qquad \dot{\Delta}_m = \varphi(2^{-m}D).
\end{equation*}
The main interest of the Littlewood-Paley decomposition (see \eqref{eq:LPdecomposition} below) is the way the dyadic blocks interact with derivatives, and more generally with homogeneous Fourier multipliers.

\begin{lemma}
Let $\sigma$ be a homogeneous function of degree $N \in \mathbb{Z}$. There exists a constant depending only on $\sigma$ and on the the dyadic decomposition function $\chi$ such that, for all $p \in [1, +\infty]$,
\begin{equation*}
\forall m \in \mathbb{Z}, \forall f \in L^p, \qquad \| \dot{\Delta}_m \sigma (D) f \|_{L^p} \leq C 2^{mN} \| \dot{\Delta}_m f \|_{L^p}.
\end{equation*}
\end{lemma}

Knowledge of the dyadic blocks formally allows to reconstruct any function: this is the Littlewood-Paley decomposition. Since the $\varphi_j$ form a partition of unity in $\R^d$, we formally have
\begin{equation}\label{eq:LPdecomposition}
{\rm Id} = \sum_{j \in \mathbb{Z}} \dot{\Delta}_j.
\end{equation}
However this identity cannot hold on the whole space $\mc S'$, as it obviously fails for polynomial functions\footnote{Inthe sequel, we will construct explicit examples of non-polynomial functions for which the Littlewood-Paley decomposition fails.} (but not only!). The questions we ask in this article are stongly linked to whether the series in \eqref{eq:LPdecomposition} converge or not for a given $f \in \mc S'$. In fact, the space $\mc S'_h$ is precisely defined as the subspace of $\mc S'$ for which decomposition \eqref{eq:LPdecomposition} holds.

\begin{defi}
Define $\mc S'_h$ as the space of those $f \in \mc S'$ such that 
\begin{equation}\label{eq:DefSh}
\chi (\lambda \xi) \what{f}(\xi) \tend_{\lambda \rightarrow + \infty} 0 \qquad \text{in } \mc S',
\end{equation}
where $\chi$ is the low frequency cut-off function defined above.
\end{defi}

\begin{rmk}
Multiple definitions of the space $\mc S'_h$ coexist. For example, the convergence \eqref{eq:DefSh} is sometimes required to be in the norm topology of $L^\infty$, as in \cite{BCD}. We follow Section 1.5.1 in \cite{Danchin-cours} and Definition 2.1.1 in \cite{Chemin-cours} and give a definition which is adapted to our context. We will give other definitions and discuss their equivalence in Section \ref{s:c2EQSpH}, as well as provide a number of examples of $\mc S'_h$ functions.
\end{rmk}

\subsection{Homogeneous Besov Spaces}

In this paragraph, we define the other object of interest for this paper: homogeneous Besov spaces. These are based on the homogeneous Littlewood-Paley decomposition \eqref{eq:LPdecomposition}.

\begin{defi}\label{d:c2SpH}
We define the \emph{homogeneous Besov space} $\dot{B}^s_{p, r}$ as the set of those classes of distributions modulo polynomials $f \in \mc S' / \, \R[X]$ such that
\begin{equation*}
\| f \|_{\dot{B}^s_{p, r}} = \left\| \left( 2^{ms} \| \dot{\Delta}_m f \|_{L^p} \right)_{m \in \mathbb{Z}} \right\|_{\ell^r(\mathbb{Z})} < +\infty.
\end{equation*}
\end{defi}

Using spaces of distributions defined modulo polynomials is nigh on impossible in certain contexts such as non-linear PDEs. In those situations, it is far better to work with givien representatives. These are readily available under certain conditions: observe that, thanks to the Bernstein inequalities (see Lemma 2.1 in \cite{BCD}), we have, for any $p \in [1, + \infty]$ and $s \in \R$,
\begin{equation*}
\sum_{- \infty}^{0} \big\| \dot{\Delta}_j f \big\|_{L^\infty} \leq C \sum_{- \infty}^0 \big\| \dot{\Delta}_j f \big\|_{L^p} 2^{jd/p} \leq C \sum_{- \infty}^0 2^{js} \big\| \dot{\Delta}_j f \big\|_{L^p} 2^{j(-s + d/p)},
\end{equation*}
so that the sum is convergent if $f \in \B$ under the condition 
\begin{equation}\label{eq:superCrit}
s < \frac{d}{p} \qquad \text{or} \qquad s = \frac{d}{p} \text{ and } r = 1.
\end{equation}
In that case, the Littlewood-Paley decomposition \eqref{eq:LPdecomposition} is convergent, the high frequency $j \geq 1$ part having a limit in $\mc S'$, and the low frequency part $j \leq 0$ converging normally in $L^\infty$. The limit $\sum_{j \in \Z} \dot{\Delta}_j f$ lies in the space $\mc S'_h$ and is equal to $f$ up to the addition of a polynomial. Therefore, the homogeneous Littlewood-Paley decomposition defines an isomorphism of Banach spaces\footnote{In other words, a linear isomorphism that is bounded together with its inverse.} onto
\begin{equation*}
\dot{\mathfrak{B}}^s_{p, r} = \Big\{ f \in \mc S'_h, \quad \| f \|_{\B} < + \infty \Big\}.
\end{equation*}
Conversely, under condition \eqref{eq:superCrit}, any element $f \in \dot{\mathfrak{B}}^s_{p, r}$ can be assigned a representative modulo polynomials in $\B$.

\medskip

The space $\dot{\mathfrak{B}}^s_{p, r}$ is called a realization of $\B$ as a subspace of $\mc S'_h$. While we do not dwell on the possibility of realizing homogeneous Besov spaces, we refer to the work \cite{Bourdaud} of Bourdaud for a presentation of the topic.

\medskip

In the case where condition \eqref{eq:superCrit} does not hold, that is for supercritical Besov spaces, there is no reason for the Littlewood-Paley decomposition of an element of $\B$ to converge, and $\B$ can no longer be realized as a subspace of $\mc S'_h$. We will prove this fact precisely in the sequel, in addition to examining how different $\dot{\mathfrak{B}}^s_{p, r}$ may be from $\B$.

\medskip

A fact that will be useful later on is that, for fixed exponents $p, r \in [1, + \infty]$, homogeneous Besov spaces of all regularity exponents are isomorphic (see Theorem 3.17 in \cite{YS}).
\begin{prop}\label{p:c2LapIsom}
Let $s, \sigma \in \R$ and $p, r \in [1, +\infty]$. Then the fractional Laplace operator defines an isomorphism of Banach spaces $(- \Delta)^{\sigma/2} : \B \tend \dot{B}^{s - \sigma}_{p, r}$.
\end{prop}

\medskip

We end this paragraph by stating duality properties which the Besov spaces inherit from the Lebesgue ones. We refer to \cite{Peetre}, Theorem 12 in Chapter 3 pp. 74-75 for a proof.\begin{thm}\label{t:c2BesovDuality}
Let $s \in \R$ and $p, r \in [1, +\infty[$. Then the topological dual of $\B$ is isomorphic to $\dot{B}^{-s}_{p', r'}$ as a Banach space, where $p'$ and $r'$ are the conjugated exponents of $p$ and $r$.
\end{thm}

\section{General Remarks on $\mc S'_h$}\label{s:c2EQSpH}

When we introduced the space $\mc S'_h$ in Definition \ref{d:c2SpH}, we noted that many different definitions coexist. This paragraph is aimed at discussing the various definitions and the way they interact. We give four of them:

\begin{enumerate}[(i)]
\item Above, we have defined $\mc S'_h$ as being the space of tempered distributions $f \in \mc S'$ that fulfill a low frequency condition
\begin{equation}\label{eq:c2defSpHWeakConvergence}
\chi(\lambda D)f \tend_{\lambda \rightarrow + \infty} 0 \qquad \text{in } \mc S'.
\end{equation}
This definition is particularly interesting with regards to homogeneous Littlewood-Paley theory: the space $\mc S'_h$ is closely related to\footnote{Strictly speaking, the convergence of the Littlewood-Paley decomposition is only equivalent to the weaker condition $\chi(2^{-j} D) f \tend 0$ as $j \rightarrow - \infty$. Convergence to zero of the subsequence does not imply the full convergence when $\lambda \rightarrow + \infty$. But this really is a technicality.} the set of distributions such that the series $\sum_{j \in \Z} \dot{\Delta}_j f$ converges to $f$ in $\mc S'$.

\item Alternatively, one may require the convergence above to take place in a stronger topology: as in \cite{BCD}, one could define $\mc S'_h$ as being the set of $f \in \mc S'$ such that
\begin{equation}\label{eq:c2defSpHStrongConvergence}
\chi(\lambda D)f \tend_{\lambda \rightarrow + \infty} 0 \qquad \text{in } L^\infty.
\end{equation}
The purpose of this definition is mainly to capture realizations of subcritical homogeneous Besov spaces: if the space $\B$ is subcritical, that is if the triplet $(s, p, r)$ satisfies \eqref{eq:superCrit}, then, for all $f \in \B$, the series $\sum_{j \leq 0} \dot{\Delta} f$ converges in $L^\infty$ and the function $\sum_{j \in \Z} \dot{\Delta}_j f = f \; ({\rm mod} \;  \R[X])$ satisfies \eqref{eq:c2defSpHStrongConvergence}.

\item Bourdaud \cite{Bourdaud} introduced the set of distributions that tend to zero at infinity: that is all the $f \in \mc D'$ such that $f(\lambda x) \tend 0$ as $\lambda \rightarrow + \infty$ and in the $\mc D'$ topology. In other words, for all $\phi \in \mc D$,
\begin{equation*}
\left\langle f(x), \frac{1}{\lambda^d} \phi \left( \frac{x}{\lambda} \right) \right\rangle_{\mc D' \times \mc D} \tend \; 0 \qquad \text{as } \lambda \rightarrow + \infty.
\end{equation*}
The intuitive meaning of this convergence is that $f$ has no ``average value''. For instance, any compactly supported distribution tends to zero at infinity. In \cite{Bourdaud}, it is shown that $\dot{B}^0_{\infty, 1}$ can be realized as a space of distributions tending to zero at infinity, and therefore so can all subcritical homogeneous Besov spaces.

\item Finally, we may impose on a $f \in \mc S'$ a condition using the heat kernel:
\begin{equation}\label{eq:c2DefSpHHeatKernel}
e^{t \Delta} f \tend_{t \rightarrow + \infty} 0 \qquad \text{in } \mc S'.
\end{equation}
This condition, although phrased slightly differently, is very similar to \eqref{eq:c2defSpHWeakConvergence}, the difference being that the heat kernel is not spectrally supported in a compact set. The convergence \eqref{eq:c2DefSpHHeatKernel} aims at eliminating all harmonic components from a given tempered distribution.
\end{enumerate}

\begin{ex}\label{ex:c1TFloc}
Let us give a few examples. First of all, any tempered distribution whose Fourier transform is integrable in a neighborhood of $\xi = 0$ is in $\mc S'_h$ according to (ii), which is the strongest of the definitions above: let $f \in \mc S'$ be such a distribution, because $\chi(\lambda \xi)$ is supported in a ball of radius $O(\lambda^{-1})$, we have
\begin{equation*}
\chi(\lambda \xi) \what{f}(\xi) \tend_{\lambda \rightarrow + \infty} 0 \qquad \text{in } L^1,
\end{equation*}
so that $\chi(\lambda D)f$ converges to zero in $L^\infty$. For example, any trigonometric polynomial with zero average value lies in $\mc S'_h$ according to definition (ii).

More generally, if $\what{f}$ is equal, on a neighborhood of $\xi = 0$, to a finite measure with no pure point component, then we also have $f \in \mc S'_h$ according to definition (ii).
\end{ex}

\begin{ex}\label{ex:c2C0}
Next, consider the space $C_0$ of continuous functions that tend to zero at $|x| \rightarrow + \infty$. Then $C_0 \subset \mc S'_h$, again according to definition (ii): let $f \in \mc C_0$ and $\epsilon > 0$, we may fix a $R > 0$ such that $|f(x)| \leq \epsilon$ for all $|x| \geq R$. Therefore $f$ is equal to the sum $f = g + h$ of a compactly supported function $g$ and function $h$ whose $L^\infty$ norm is smaller than $\| h \|_{L^\infty} \leq  \epsilon$, and so, by Example \ref{ex:c1TFloc} above,
\begin{equation*}
\| \chi(\lambda D)f(x) \|_{L^\infty} = \| \chi(\lambda D)h(x) \|_{L^\infty} + o(1) \lesssim \epsilon + o(1).
\end{equation*}
Now, since $C_0 \subset \mc S'_h$ in the sense of (ii), it is also true in the sense of (i).
\end{ex}

\begin{ex}\label{ex:c2Sign}
In contrast with the two above, another (more subtle) example may help us point out what differences exist between the various definitions above. Let $\sigma = \mathds{1}_{\R_+} - \mathds{1}_{\R_-}$ be the sign function. Then we have
\begin{equation*}
\chi(\lambda D) \sigma (x) = \int_{- \infty}^{+ \infty} \sigma(x - y) \psi_\lambda(y) {\rm d}y.
\end{equation*}
Here and in the sequel, we use the following notation: $\psi \in \mc S$ is a function such that $\what{\psi} =\chi$ and, for any function $g$ and $\lambda > 0$, we set $g_\lambda(x) = \lambda^{d} g(\lambda x)$. This form of $\chi(\lambda D)\sigma$ shows that the sign function cannot possibly satisfy \eqref{eq:c2defSpHStrongConvergence}, as dominated convergence provides\footnote{Recall that $\int \psi = \what{\psi}(0) = \chi(0) = 1$.} the limit $\psi_\lambda * \sigma (x) \tend \pm 1$ as $x \rightarrow \pm \infty$, so $\| \chi(\lambda D)\sigma \|_{L^\infty} = 1$. On the other hand, $\psi_\lambda * \sigma$ tends to zero uniformly locally\footnote{In other words, for every compact $K \subset \R^d$, the functions $(\psi_\lambda * \sigma) \mathds{1}_K$ tend to zero in $L^\infty$.}, and so it does in $\mc S'$. The same argument applies to show that $e^{t \Delta} \sigma \tend 0$ as $t \rightarrow + \infty$ (in $\mc S'$). Finally, since $\sigma$ is an odd function and $\psi_\lambda$ an even one, we have
\begin{equation*}
\big\langle \sigma(x), \psi_\lambda (-x) \big\rangle_{L^\infty \times L^1}  = 0.
\end{equation*}
However, it must be noted that despite the previous cancellation, the sign function $\sigma$ does not tend to zero at infinity in the sense of (iii). Taking a nonzero test function $\phi \in \mc D$ such that $\phi \leq 0$ in $\R_-$ and $\phi \geq 0$ in $\R_+$ gives
\begin{equation*}
\int_{- \infty}^{+ \infty} \sigma (\lambda x) \phi(x) \dx = \int_{- \infty}^{+ \infty} |\phi| \neq 0.
\end{equation*}
\end{ex}

\medskip

We will study the way definitions (i) and (iii) interact in the special case of $L^\infty(\R^d)$ with $d \geq 1$.

\begin{prop}\label{p:c2SpHequiv}
Consider $f \in L^\infty$. Consider $\psi \in \mc S$ such that $\what{\psi} = \chi$ and define, for $\lambda > 0$, the function $\psi_\lambda (x) = \lambda^{-d} \psi(\lambda^{-1}x)$. The following assertions are equivalent:
\begin{enumerate}[(i)]
\item we have $\chi (\lambda D) f \stackrel{*}{\wtend} 0$ in $L^\infty$, as $\lambda \rightarrow + \infty$;
\item we have $\big\langle f, \psi_\lambda \big\rangle_{L^\infty \times L^1} \tend 0$ as $\lambda \rightarrow + \infty$.
\end{enumerate}
\end{prop}

\begin{proof}
The link between the two statements is this: the brackets of \textit{(ii)} can be seen as a particular value of a convolution product, namely (recall that $\chi$ and $\psi$ are radial functions)
\begin{equation*}
\psi_\lambda * f(0) = \left\langle f(y), \frac{1}{\lambda^d} \psi \left( \frac{-y}{\lambda} \right) \right\rangle_{L^\infty \times L^1}.
\end{equation*}
All we have to do is show that the value of $\psi_\lambda * f(x)$ cannot be too far from $\psi_\lambda * f(0)$. This is a consequence of the regularizing nature of $\psi_\lambda$. Since $\chi(\lambda D)$ is a low frequency cut-off, the function $\chi (\lambda D)f = \psi_\lambda * f$ is smooth (analytic in fact) with good estimates on its derivatives. Thus, a Taylor expansion is an appropriate way to study the difference between the value at $x$ and at zero of the convolution product:
\begin{equation*}
\begin{split}
\psi_\lambda * f(x) & = \psi_\lambda * f(0) + x \cdot \int_0^1 \nabla \psi_\lambda * f(tx) {\rm d}t = \psi_\lambda * f(0) + \frac{x}{\lambda} \cdot \int_0^1 (\nabla \psi)_\lambda * f(tx) {\rm d}t \\
& = \psi_\lambda * f(0) + O \left( \frac{|x|}{\lambda} \right),
\end{split}
\end{equation*}
where the constant in the $O(\, . \, )$ is $\| \nabla \psi \|_{L^1} \| f \|_{L^\infty}$. On the one hand, if \textit{(ii)} holds, then the previous equation shows that $\psi_\lambda * f$ converges locally uniformly to zero, thus giving \textit{(i)}. On the other, by fixing $\phi \in \mc S$, we obtain
\begin{equation*}
\left| \big\langle \psi_\lambda * f, \phi \big\rangle_{L^\infty \times L^1} - \psi_\lambda * f (0) \int \phi \; \right| \leq \frac{1}{\lambda} \| \nabla \psi \|_{L^1} \| f \|_{L^\infty} \int |x| |\phi (x)| \dx = O \left( \frac{1}{\lambda} \right),
\end{equation*}
so that the convergence of $\psi_\lambda * f(0)$ implies weak convergence of $\chi(\lambda D)f$ to zero.
\end{proof}

\begin{rmk}
Proposition \ref{p:c2SpHequiv} hints as to which ones of the definitions (i)-(iv) above are relatively stronger. For $f \in L^\infty$, then (ii) implies (iii), which implies both (i) and (iv).
\end{rmk}

\section{Closure of $\mc S'_h$ in Besov Spaces}\label{s:c1ClosB}

In this paragraph, we focus on the topological properties of the intersection $\mc S'_h \cap \B$. The following Proposition seems to be common knowledge (especially point \textit{(i)}), although we have been unable to locate a proof in the literature.

We point out that assertion \textit{(ii)} in Theorem \ref{t:c2closB} below finds a close counterpart in Proposition 2.27 and Remark 2.28 in \cite{BCD}. However, the authors of \cite{BCD} use a different definition of $\mc S'_h$.

\begin{thm}\label{t:c2closB}
Consider $(s, p, r) \in \mathbb{R}\times [1, + \infty]^2$ such that $\B$ is supercritical, that is $s > d/p$, or $s=d/p$ and $r > 1$, then the following assertions hold:
\begin{enumerate}[(i)]
\item the subspace $\mc S'_h \cap \B$ is not closed in $\B$;
\item the intersection $\mc S'_h \cap \B$ is dense in $\B$ if and only if $r < + \infty$.
\end{enumerate}
\end{thm}

\begin{rmk}
In particular, with the above choice of exponents $(s, p, r)$, the space $\B$ cannot be realized as a subspace of $\mc S'_h$. In other words, there is no linear map $\sigma : \B \tend E$ to a subspace $E \subset \mc S'$ such that $E \subset \mc S'_h$ and the following diagram commutes:
\begin{equation*}
\xymatrix{
\B \ar[rr]^\subset \ar[dd]^\sigma & & \mc S' / \, \R[X] \\
& & \\
E \ar[rruu]_\pi
}
\end{equation*}
where in the above $\pi : E \subset \mc S' \tend \mc S' / \, \R[X]$ is the natural projection. Indeed, if that were the case, any function $f \in \B$ with $f \notin \mc S'_h \cap \B$ would define an element $\sigma(f) \in \mc S'_h$ which would be mapped to an element of $\mc S'_h \cap \B$ by $\pi$, this being a contradiction.
\end{rmk}

\begin{proof}[Proof (of Theorem \ref{t:c2closB})]
We start by showing point \textit{(i)}. The idea of the proof is to exhibit an element $\B$ which does not belong to $\mc S'_h$ (that is to its image in $\mc S' / \, \R[X]$). Let $\psi \in \mc S$ be a nonzero function with nonnegative Fourier transform $\what{\psi} \geq 0$ such that $\varphi(\xi) \what{\psi}(\xi) = \what{\psi}(\xi)$, where $\varphi$ is the Littlewood-Paley decomposition function \eqref{eq:LPdecomposition}, and define $\psi_j(x) = 2^{jd}\psi(2^{j}x)$ for $j \in \Z$ so that
\begin{equation}\label{eq:smallLPblock}
\dot{\Delta}_j \psi_j = \psi_j \qquad \text{and} \qquad \| \psi_j \|_{L^p} = 2^{jd\left( 1 - \frac{1}{p} \right)} \| \psi \|_{L^p}
\end{equation}
for all $p \in [1, + \infty]$. We define our function by
\begin{equation}\label{eq:CounterEx}
g = \sum_{- \infty}^{-1} 2^{-jd\left(1 - \frac{1}{p} \right)} 2^{-js} \frac{1}{|j|^\alpha} \psi_j,
\end{equation}
where $\alpha \in [1, + \infty]$ is chosen so that
\begin{equation*}
\alpha = 2 \text{ if } s > \frac{d}{p} \qquad \text{and} \qquad \alpha = 1 \text{ if } s = \frac{d}{p}.
\end{equation*}
Now, the Besov norm of $g$ is finite: we compute
\begin{equation*}
\begin{split}
\| g \|_{\B} & \approx \left( \sum_{-\infty}^{-1} \left( 2^{-jd\left(1 - \frac{1}{p} \right)} \| \psi_j \|_{L^p} \right)^r \frac{1}{|j|^{\alpha r}} \right)^{1/r} \\
& \approx \left( \sum_{- \infty}^{-1} \frac{1}{|j|^{\alpha r}} \right)^{1/r},
\end{split}
\end{equation*}
with the usual modification of taking the $\ell^\infty (\Z)$ norm if $r = + \infty$. By remembering that $r > 1$ if $s = d/p$, we see that this last sum is finite. In particular, we see that the series \eqref{eq:CounterEx} defining $g$ converges in $\B$, which is a Banach space. Therefore \eqref{eq:CounterEx} does indeed define an element of $\B$.

\medskip

We must now prove that $g \notin \mc S'_h \cap \B$, or in other words that the series defining $g$ does not converge in the $\mc S'$ topology. For this, we fix a $\phi \in \mc S$ such that $\what{\phi}(\xi) \equiv 1$ around $\xi = 0$ and we compute the partial sum
\begin{equation*}
\begin{split}
\left\langle \sum_{- N}^{-1} 2^{j-d\left(1 - \frac{1}{p} \right)} 2^{-js} \frac{1}{|j|^\alpha} \psi_j, \phi \right\rangle_{\mc S' \times \mc S} & = \sum_{-N}^{-1} 2^{j(d/p - s)} \frac{1}{|j|^\alpha} \int \what{\psi} (2^{-j}\xi) 2^{-jd} {\rm d} \xi \\
& \approx \sum_{-N}^{-1} 2^{j(d/p - s)} \frac{1}{|j|^\alpha}.
\end{split}
\end{equation*}
By choice of $s$ and $\alpha$, this last sum diverges as $N \rightarrow + \infty$, and therefore $g$ cannot be an element of $\mc S'_h \cap \B$. However, since the series \eqref{eq:CounterEx} is convergent in $\B$, the function $g$ is in the closure of $\mc S'_h \cap \B$, and so the intersection is not closed.

\medskip

We now get to point \textit{(ii)} of the Proposition. We focus on the case $r < +\infty$, since the case of $r = + \infty$ will be an immediate consequence of Theorem \ref{t:c2nonSepB} below (whose proof is entirely independent). It is simply a matter of noting that for any $f \in \B$, the series
\begin{equation*}
f_N := \sum_{-N}^\infty \dot{\Delta}_j f
\end{equation*}
lies in $\mc S'_h \cap \B$ and converges to $f$ in $\B$.
\end{proof}

\section{Non-Separability of the Quotient: the Case of Besov Spaces}\label{s:c1SepB}

Given Theorem \ref{t:c2closB} above, we see that the inclusion $\mc S'_h \cap \B \subsetneq \B$ is very nearly an equality if $r < + \infty$, but when $r = +\infty$ we have not yet examined the difference between $\mc S'_h \cap \BB$ and $\BB$. We make the following definition.

\begin{defi}
For any $p \in [1, +\infty]$ and $s \geq d/p$, we define $\CC$\index{$\Csp$@$\CC$} to be the closure of $\mc S'_h \cap \BB$ in the $\BB$ topology.
\end{defi}

At this point, we must note the importance the precise definition of $\mc S'_h$ plays in our study. When $\mc S'_h$ is defined in terms of the norm topology of $L^\infty$, as in point (ii) at the beginning of Section \ref{s:c2EQSpH}, then $\CC$ is the set of $f \in \BB$ such that
\begin{equation*}
2^{js} \dot{\Delta}_j f \tend_{j \rightarrow - \infty} 0 \qquad \text{in } L^\infty,
\end{equation*}
as explained in Remark 2.28 of \cite{BCD}. Things are not the same in our framework (see Definition \ref{d:c2SpH} above). We may be inspired by Example \ref{ex:c2Sign} to exhibit an element of $\mc S'_h$ such that the convergence above does not hold. Let $\sigma$ be, as in Example \ref{ex:c2Sign}, the sign function on $\R$. Then, by noting $\phi_j (x) = 2^j \phi(2^j x)$ the kernel of the operator $\dot{\Delta}_j$, we may change variables in the convolution integral to obtain
\begin{equation*}
\dot{\Delta}_j \sigma(x) = \int_{- \infty}^{+ \infty} \sigma(y) \phi(2^j x - 2^j y) 2^j {\rm d} y = \int_{- \infty}^{+ \infty} \sigma(y) \phi(2^j x - y) {\rm d} y = \dot{\Delta}_1 \sigma (2^j x),
\end{equation*}
so that $\| \dot{\Delta}_j \sigma \|_{L^\infty} = \| \dot{\Delta}_1 \sigma \|_{L^\infty} > 0$. As a consequence, the sign function defines an element of $\mc S'_h \cap \dot{B}^0_{\infty, \infty} \subset C^0_{\infty}$ for which the convergence above does not hold for the strong $L^\infty$ topology.

\medskip

The next Theorem states that the inclusion $\CC \subsetneq \BB$ is strict. In fact, Theorem \ref{t:c2nonSepB} does more than that, as it expresses the strict inclusion in terms of the size of the quotient $\BB / \CC$, which is not separable.

\begin{thm}\label{t:c2nonSepB}
Let $p \in [1, + \infty]$ and $s \geq d/p$. The quotient space $\BB / \CC$ is not separable. In fact, there exists an embedding $J : \ell^\infty \tend \BB$ which defines a quasi-isometry between the (non-separable) quotient spaces $J' : \ell^\infty / c_0 \tend \BB / \CC$.
\end{thm}

\begin{proof}
We start by defining an embedding $J : \ell^\infty \tend \BB$. For any $u \in \ell^\infty$, we formally define
\begin{equation}\label{eq:c2Jappl}
Ju = \sum_{- \infty}^0 u(-j) 2^{-jd\left(1 - \frac{1}{p} \right)} 2^{-js} \psi_j,
\end{equation}
where the functions $\psi_j$ are as in \eqref{eq:smallLPblock} above. Now, unlike the series in \eqref{eq:CounterEx}, it is not at all obvious that $Ju$ defines an element of $\BB$ because the sum \eqref{eq:c2Jappl} need not converge in $\BB$ if $u \notin c_0$. Instead, to clarify the meaning of \eqref{eq:c2Jappl}, we fix $\sigma < d/p$ and set\footnote{We cannot simply take $\sigma = 0$, it would be insufficient when $p = + \infty$.}
\begin{equation*}
g = (- \Delta)^{s - \sigma} Ju = \sum_{- \infty}^0 u(-j) 2^{jd\left(1 - \frac{1}{p} \right)} 2^{-js} (- \Delta)^{s -\sigma} \psi_j.
\end{equation*}
Though the sum above does not converges in the Besov space $\dot{B}^\sigma_{p, \infty}$ any more than it did in $\BB$, the \textsl{supercriticality} of $\dot{B}^\sigma_{p, \infty} \subset \dot{B}^{\sigma - d/p}_{\infty, \infty}$ implies it must converge in $\mc S'$, and so defines a distribution $g \in \mc S'$ with a finite $\dot{B}^\sigma_{p, \infty}$ norm and therefore an element of $\dot{B}^\sigma_{p, \infty}$. Since the fractional Laplacian defines a quasi-isometry (see Proposition \ref{p:c2LapIsom} above)
\begin{equation*}
\xymatrix{
(- \Delta)^{\sigma - s} : \dot{B}^\sigma_{p, \infty} \ar[r] & \BB,
}
\end{equation*}
we can in turn define $J$ by the formula $Ju := (- \Delta)^{\sigma - s} g$.

\medskip

We now are ready to define our map $J' : \ell^\infty / c_0 \tend \BB / \CC$. First of all, we note that if $u \in c_0$ then the series \eqref{eq:c2Jappl} is convergent in the $\BB$ topology, so that $Ju$ is a $\BB$ limit of functions whose Fourier transform is supported away from $\xi = 0$. We deduce that $J(c_0) \subset \CC$. As a consequence, we may define a quotient map $J'$ such that the following diagram commutes:
\begin{equation*}
\xymatrix{
\ell^\infty \ar[rr]^J \ar@{->>}[d] & & \BB \ar@{->>}[d] \\
\ell^\infty / c_0 \ar[rr]^{J'} & & \BB / \CC
}
\end{equation*}
where the vertical maps are the natural projections. To conclude, we must show that $J'$ is a quasi-isometry. In order to do so, we will prove that the functions of $\CC$ inherit a low frequency property from $\mc S'_h \cap \BB$ which $Ju$ cannot posses if $u \notin c_0$. More precisely, we prove that if $g \in L^1$ and $f \in \BB$, then we have:
\begin{equation}\label{eq:c2IneqQuot}
\limss_{j \rightarrow - \infty} \left| \left\langle 2^{j(s - d/p)}\dot{\Delta}_j f, g \right\rangle_{L^\infty \times L^1} \right| \leq \| g \|_{L^1} {\rm dist} (f, \CC) = \| g \|_{L^1} \inf_{h \in \CC} \| f-h \|_{\BB}.
\end{equation}
Taking inequality \eqref{eq:c2IneqQuot} for granted and leaving its proof for later, we are nearly done: by using \eqref{eq:smallLPblock}, we see that the dyadic blocks of $Ju$ are
\begin{equation*}
2^{j(s - d/p)} \dot{\Delta}_j Ju(x) = u(-j) \psi(2^j x)
\end{equation*}
and so, by dominated convergence,
\begin{equation*}
\left\langle 2^{j(s - d/p)} \dot{\Delta}_j Ju, g \right\rangle_{L^\infty \times L^1} \sim u(-j) \, \psi(0) \int g \qquad \text{as } j \rightarrow - \infty.
\end{equation*}
Finally, by choosing $g\in \mc S$ such that $\int g \neq 0$ and noting that, on the one hand $\psi(0) = \int \what{\psi} > 0$, and on the other that
\begin{equation}\label{eq:c1SeqAssert}
\limss_{j \rightarrow - \infty} |u(-j)| \, = \, \inf_{w \in c_0} \|u - w\|_{\ell^\infty} \, = \, \| u \|_{\ell^\infty / c_0},
\end{equation}
we see that $J'$ is indeed a quasi-isometry. Let us prove this last assertion \eqref{eq:c1SeqAssert}. Firstly, it is clear that for all $w \in c_0$, we must have
\begin{equation*}
\limss_{j \rightarrow - \infty} |u(-j)| \, = \, \limss_{j \rightarrow - \infty} |u(-j) - w(-j)| \leq \inf_{w \in c_0} \| u - w \|_{\ell^\infty}.
\end{equation*}
In order to get the reverse inequality, we use the definition of the limit superior as an infimum of suprema. We have
\begin{equation*}
\begin{split}
\limss_{j \rightarrow - \infty} |u(-j)| \, & =\, \inf_{J \rightarrow - \infty} \sup_{j \leq J} \big| u(-j) \big| \\
& =\,  \inf_{J \rightarrow - \infty} \| u - \mathds{1}_{[-J, 0]} u \|_{\ell^\infty} \geq \inf_{w \in c_0} \| u-w \|_{\ell^\infty},
\end{split}
\end{equation*}
because the sequence $\mathds{1}_{[-J, 0]} u$ is finitely supported, and so must lie in $c_0$. Both inequalities prove that equation \eqref{eq:c1SeqAssert} holds.

\medskip

It only remains to prove \eqref{eq:c2IneqQuot}. First of all, we remark that if $h \in \mc S'_h \cap \BB$ and $g \in \mc S$, the condition $s \geq d/p$ implies that for all $j \leq 0$,
\begin{equation*}
\left| \left\langle 2^{j(s - d/p)}\dot{\Delta}_j h, g \right\rangle_{L^\infty \times L^1} \right| \leq \left| \left\langle \dot{\Delta}_j h, g \right\rangle_{\mc S' \times \mc S} \right| \tend_{j \rightarrow - \infty} 0.
\end{equation*}
Next, if $h \in \CC$, we may fix a sequence of functions $h_k \in \mc S'_h \cap \BB$ that converge to $h$ in $\BB$. Then, by using the Bernstein inequalities, we obtain
\begin{equation}\label{eq:c2IneqQuotPrep}
\begin{split}
\left| \left\langle 2^{j(s - d/p)}\dot{\Delta}_j h, g \right\rangle_{L^\infty \times L^1} \right| & \leq 2^{j(s - d/p)} \| \dot{\Delta}_j (h - h_k) \|_{L^\infty} \| g \|_{L^1} \\
& \qquad \qquad \quad \quad \qquad + \left| \left\langle 2^{j(s - d/p)}\dot{\Delta}_j h_k, g \right\rangle_{L^\infty \times L^1} \right| \\
& \leq \| h - h_k \|_{\BB} \| g \|_{L^1} + \left| \left\langle 2^{j(s - d/p)}\dot{\Delta}_j h_k, g \right\rangle_{L^\infty \times L^1} \right|
\end{split}
\end{equation}
and so
\begin{equation*}
\limss_{j \rightarrow - \infty} \left| \left\langle 2^{j(s - d/p)}\dot{\Delta}_j h, g \right\rangle_{L^\infty \times L^1} \right| \leq \| h - h_k \|_{\BB} \| g \|_{L^1} \tend_{k \rightarrow + \infty} 0.
\end{equation*}
which implies that the lefthand side of this last inequality must be zero. Finally, by proceeding exactly as in \eqref{eq:c2IneqQuotPrep}, we see that for all $f \in \BB$ and all $h \in \CC$, we have a similar inequality
\begin{equation*}
\limss_{j \rightarrow - \infty} \left| \left\langle 2^{j(s - d/p)}\dot{\Delta}_j f, g \right\rangle_{L^\infty \times L^1} \right| \leq \| f-h \|_{\BB} \| g \|_{L^1},
\end{equation*}
which gives in turn \eqref{eq:c2IneqQuot}.
\end{proof}

\section{Non-Complementation of the Closure $\CC$}\label{s:c1ComplB}

In this paragraph, we reach the core of our study on $\BB$ and $\CC$. The aim of Theorem \ref{t:c2ComplB} is to reveal the role of $\CC$ in the structure of $\BB$. Precisely, we show that $\CC$ is not complemented in $\BB$: there is no continuous projection $P : \BB \tend \BB$ with range exactly $\CC$. The spirit of this property is to show how different $\CC$ and $\BB$ really are: the proof of Theorem \ref{t:c2ComplB} heavily relies  on the the fact that $\BB / \CC$ is very large (in fact large enough to contain an isomorphic copy of $\ell^\infty / c_0$).

\begin{thm}\label{t:c2ComplB}
Let $p \in [1, + \infty]$ and $s \geq d/p$. Then $\CC$ is not complemented in $\BB$: there is no decomposition $\BB = \CC \oplus G$ with continuous projections.
\end{thm}

The spirit of what follows is to adapt the ideas of Whitley \cite{W} (see also \cite{AK}, Section 2.5) in his proof of the Phillips-Sobczyk theorem, which states that $c_0$ is not complemented in $\ell^\infty$. Analysis of Whitley's proof, which is based on a countability argument, reveals two key features which we will need to adapt in the framework of Besov spaces:
\begin{enumerate}[(i)]
\item the existence of an uncountable family of subspaces $\ell^\infty(A_i) \subset \ell^\infty (\N)$, for $i \in I$, that are not in $c_0$ and such that the intersection of any two of these spaces is in $c_0$; in other words, they are mutually independent up to elements of $c_0$;
\item the fact that the separation of points can be tested by a countable set of equalities: for all $u \in \ell^\infty$, we have $u = 0$ if and only if $u(n) = 0$ for all $n \in \N$.
\end{enumerate}
While both these facts seem very specific to $\ell^\infty$, we will find homologous assertions in $\BB$. Firstly, we will see that the embedding $J : \ell^\infty \tend \BB$ of Theorem \ref{t:c2nonSepB} preserves the properties of the spaces $\ell^\infty(A_i)$ used in Whitley's argument. Secondly, the space $\BB$ has a separable\footnote{Recall from Theorem \ref{t:c2BesovDuality} that we note $p'$ the conjugate Lebesgue exponent $\frac{1}{p} + \frac{1}{p'} = 1$.} predual space $\dot{B}^{-s}_{p', 1}$ for all $p > 1$ (see Theorem \ref{t:c2BesovDuality} above), and so the separation of points can be tested by a countable number of equalities (the case $p=1$ will recieve special attention).

\medskip

\textbf{STEP 1}. We begin by proving Theorem \ref{t:c2ComplB} in the case where $p > 1$. The argument in the case $p=1$ is summarized in Corollary \ref{c:p1case} below, whose proof is entirely independent.

We start by constructing an uncountable family of subspaces of $\BB$ such that the intersection of any two of these spaces lies in $\CC$. The existence of such spaces will stem from the following Lemma (see for example Lemma 2.5.3 in \cite{AK}), which we reproduce and prove for the reader's convenience.

\begin{lemma}\label{l:c2countArg}
There exists an uncountable family $(A_i)_{i \in I}$ of infinite subsets of $\N$ such that, for any two $i \neq j$, there is a finite intersection $|A_i \cap A_j| < + \infty$.
\end{lemma}

\begin{proof}[Proof (of the Lemma)]
Since only countability matters in the statement we wish to prove, nothing is lost in replacing $\N$ by $\Q$ and seeking the $A_i$ as subsets of $\Q$. Next, for any irrational $\theta \in \R \setminus \mathbb{Q}$, fix a sequence $(q_k)$ of rational numbers such that $q_k \rightarrow \theta$.

Define $A_\theta = \{ q_k, \; k\geq 0 \}$. Then the sets $(A_\theta)_{\theta \in \R}$ are all infinite and any two of these sets must have finite intersection.
\end{proof}

In particular, the subspaces $\ell^\infty(A_i)$ of $\ell^\infty$ sequences which are compactly supported in $A_i$ have the following properties: for $i \neq j$,
\begin{equation*}
\ell^\infty(A_i) \nsubseteq c_0 \qquad \text{and} \qquad \ell^\infty(A_i) \cap \ell^\infty(A_j) \subset c_0.
\end{equation*}
In what follows, we will transport these spaces into $\BB$ by means of a well chosen map: recall $J : \ell^\infty \tend \BB$ from Theorem \ref{t:c2nonSepB}, which we have seen to satisfy $J^{-1}(\CC) = c_0$ so that  $Ju \in \CC$ if and only if $u \in c_0$. This implies that the image spaces $J \big( \ell^\infty (A_i) \big)$ satisfy
\begin{equation*}
J\big( \ell^\infty(A_i) \big) \nsubseteq \CC \qquad \text{and} \qquad J \big( \ell^\infty(A_i) \cap \ell^\infty(A_j) \big) \subset \CC
\end{equation*}
and we see that the spaces $J\big(\ell^\infty(A_i)\big)$ will be well-suited for our purpose.

\medskip

\textbf{STEP 2}. Consider a nonzero bounded operator $T : \BB \tend \BB$ such that $\CC \subset \ker(T)$. We will prove the existence of a $i \in I$ such that $J\big(\ell^\infty(A_i)\big) \subset \ker(T)$.

Assume, in order to obtain a contradiction, that none of the $J\big(\ell^\infty(A_i)\big)$ lie in the kernel of $T$. Therefore, for every $i \in I$, we may find a $u_i \in \ell^\infty (A_i)$ such that $TJu_i \neq 0$. In addition, we may assume $u$ to be in the unit ball $\| u \|_{\ell^\infty} \leq 1$.

Next, because the predual space $\dot{B}^{-s}_{p', 1}$ of $\BB$ is separable (remember that $p > 1$ for now), we may fix a sequence $(g_n)_{n \geq 1}$ which forms a dense subset of the unit ball of $\dot{B}^{-s}_{p', 1}$. Then
\begin{equation*}
\begin{split}
I & = \{ i \in I, \; TJu_i \neq 0 \} = \bigcup_{n \geq 0} \left\{ i \in I, \quad \left\langle TJu_i, g_n \right\rangle_{\BB \times \dot{B}^{-s}_{p', 1}} \neq 0 \right\} \\
& = \bigcup_{n, k \geq 0} \left\{ i \in I, \quad \left| \left\langle TJu_i, g_n \right\rangle_{\BB \times \dot{B}^{-s}_{p', 1}} \right| \geq \frac{1}{k+1} \right\} := \bigcup_{n, k \geq 0} I_{n, k}.
\end{split}
\end{equation*}
Because $I$ is uncountable and is the countable union of the $I_{n, k}$, there must exist indices $n, k \geq 0$ such that the set $I_{n, k}$ is also uncountable: in particular, there exists an infinite number of $i \in I$ such that the bracket $\left\langle TJu_i, g_n \right\rangle_{\BB \times \dot{B}^{-s}_{p', 1}}$ is not small.

To take advantage of this last fact, we will construct a linear combination of the $u_i$ which will make the bracket become arbitrarily large. Fix a finite subset $F \subset I_{k, n}$ and define the sequence
\begin{equation*}
y = \sum_{i \in F} \alpha_i u_i,
\end{equation*}
where the $\alpha_i$ are chosen so that the bracket $\langle TJy, g_n \rangle_{\BB \times \dot{B}^{-s}_{p', 1}}$ becomes large:
\begin{equation*}
\alpha_i = \frac{\left\langle TJu_i, g_n \right\rangle_{\BB \times \dot{B}^{-s}_{p', 1}}}{|\left\langle TJu_i, g_n \right\rangle_{\BB \times \dot{B}^{-s}_{p', 1}}|} = \pm 1.
\end{equation*}
Therefore
\begin{equation}\label{eq:ComplB1}
\left|\langle TJy, g_n \rangle_{\BB \times \dot{B}^{-s}_{p', 1}} \right| \geq \frac{|F|}{k+1},
\end{equation}
and this lower bound may be made as large as desired by taking $|F|$ as large as needed, the set $I_{k, n}$ being uncountably infinite. On the other hand, because the subsets $A_i$ have finite intersection, we may decompose the union of the $A_i$ as $\cup_{i \in F} A_i = A \sqcup B$ where
\begin{equation}\label{eq:UnionAB}
B = \bigcup_{i \neq j} (A_i \cap A_j)
\end{equation}
the union ranging on all $i, j \in F$ such that $i \neq j$, is a finite set and any $m \in A$ is in exactly one of the $A_i$. By setting $a = \mathds{1}_A y$ and $b = \mathds{1}_B y$, we see that $y = a + b$ with $\| a \|_{\ell^\infty} \leq 1$ and $b$ having finite support. Since $b$ has finite support, $Jb \in \CC$ and $TJy = TJa$, so
\begin{equation}\label{eq:ComplB2}
\left| \left\langle TJu_i, g_n \right\rangle_{\BB \times \dot{B}^{-s}_{p', 1}} \right| \leq \| T \|.
\end{equation}
By comparing \eqref{eq:ComplB1} and \eqref{eq:ComplB2}, we have obtained the contradiction we were seeking, since the set $F$ can be chosen as large as desired, $I_{k, n}$ being uncountably infinite.

\medskip

\textbf{STEP 3}. We may now end the proof of Theorem \ref{t:c2ComplB}.

\begin{proof}[Proof of Theorem \ref{t:c2ComplB}]
Assume on the contrary that $\CC$ has a topological supplementary $\BB = \CC \oplus G$ and let $T : \BB \tend G$ be the associated projection on $G$. Then $T$ is a bounded operator such that $\CC = \ker(T)$ and step 2 gives a $i \in I$ such that $J\big( \ell^{\infty}(A_i) \big) \subset \CC$. But this is a contradiction: Theorem \ref{t:c2nonSepB} asserts that $Ju$ cannot lie in $\CC$ if $u \notin c_0$, which is certainly the case if $u=\mathds{1}_{A_i} \in \ell^\infty(A_i)$.
\end{proof}

\section{The Critical Case: the Intersection $\mc S'_h \cap L^\infty$}\label{s:c2Linfty}

So far in our study, it appears that the space $L^\infty$ of bounded functions plays a critical role in that it lies at the interface between the two very different behaviors the homogeneous Besov spaces have: $L^\infty$ is in the center of the chain of embeddings
\begin{equation*}
\xymatrix{
\dot{B}^{0}_{\infty, 1} \ar[r]^\subset & L^\infty \ar@{->>}[r] & L^\infty / \, \R \ar[r]^\subset & \dot{B}^{0}_{\infty, \infty},
}
\end{equation*}
while on the one hand $\dot{B}^{0}_{\infty, 1} \subset \mc S'_h$, and on the other $\dot{B}^{0}_{\infty, \infty}$ gathers all the properties described in Theorem \ref{t:c2nonSepB} and \ref{t:c2ComplB}. A very natural question is whether the space $\mc S'_h \cap L^\infty$ has an analogous role in the structure of $L^\infty$ as $C^0_\infty$ did for the Besov space $\dot{B}^0_{\infty, \infty}$.

\subsection{The Intersection is Closed}

Our first answer shows a difference between $L^\infty$ and Besov spaces. While $\mc S'_h \cap \dot{B}^0_{\infty, \infty}$ was not closed in $\dot{B}^0_{\infty, \infty}$, the intersection $\mc S'_h \cap L^\infty$ is.

\begin{prop}
The space $L^\infty \cap \mc S'_h$ is closed in $L^\infty$ for the strong topology.
\end{prop}

\begin{proof}
Let $(f_n)_{n \geq 0}$ be a converging sequence of functions in $L^\infty \cap \mc S'_h$ whose limit is $f \in L^\infty$. We have, for all $\phi \in \mc S$,
\begin{align*}
\left| \left\langle \chi(\lambda D)f, \phi \right\rangle_{L^\infty \times L^1} \right| & \leq \left| \left\langle \chi (\lambda D)(f - f_n) , \phi \right\rangle_{L^\infty \times L^1} \right| +  \left| \left\langle \chi(\lambda D) f_n, \phi \right\rangle_{L^\infty \times L^1} \right| \\
& \leq \left\| \chi (\lambda D) (f - f_n) \right\|_{L^\infty} \| \phi \|_{L^1} + \left| \left\langle \chi(\lambda D) f_n, \phi \right\rangle_{\mc S' \times \mc S} \right|.
\end{align*}
The fact that the $\chi(\lambda D)f_n$ converge to $0$ in $\mc S'$ as $\lambda \rightarrow +\infty$ shows that we have,
\begin{equation*}
\forall n \geq 0, \qquad \limss_{\lambda \rightarrow + \infty} \left| \left\langle \chi(\lambda D)f, \phi \right\rangle_{\mc S' \times \mc S} \right| \leq   C \left\| (f - f_n) \right\|_{L^\infty} \| \phi \|_{L^1}.
\end{equation*}
This term has limit $0$ as $n \rightarrow +\infty$ so that $f$ indeed lies in $L^\infty \cap \mc S'$.
\end{proof}

\subsection{Non-Separability of the Quotient}

\begin{thm}\label{t:c2nonSepL}
The quotient space $L^\infty / (\mc S'_h \cap L^\infty)$ is not separable. In fact, there exists an embedding $J : \ell^\infty \tend L^\infty$ which defines a quasi-isometry between the (non-separable) quotient spaces $J' : \ell^\infty / c_0 \tend L^\infty / (\mc S'_h \cap L^\infty)$.
\end{thm}

\begin{proof}
To construct the map $J : \ell^\infty \tend L^\infty$, the idea is to take advantage of Proposition \ref{p:c2SpHequiv} which states that the space $\mc S'_h$ is characterized by convergence of a family of average values: if $\chi$ is the Littlewood-Paley decomposition function, as in Subsection \ref{ss:c2LPdecomp} and $\psi \in \mc S$ such that $\what{\psi} = \chi$, then a bounded function $f \in L^\infty$ is in $\mc S'_h$ if and only if there is convergence of the average values
\begin{equation*}
\langle f, \psi_\lambda \rangle_{L^\infty \times L^1} = \int f(x) \, \psi \left( \frac{x}{\lambda} \right) \frac{\dx}{\lambda^d} \tend_{\lambda \rightarrow + \infty} 0.
\end{equation*}
For any $u \in \ell^\infty$, we will construct a function $Hu \in L^\infty$ whose average values $\langle f, \psi_\lambda \rangle_{L^\infty \times L^1}$ will share accumulation points with the sequence $u$ when $\lambda \rightarrow + \infty$.

\medskip

Consider an increasing sequence $(r_m)_m$ of radii (which we will fix later on) with $r_0 = 0$ and define a family of annuli by $\mc C_m = \{ r_m \leq |x| \leq r_{m+1} \}$. For every sequence $u \in \ell^\infty$, we set
\begin{equation*}
Ju = \sum_{m=0}^\infty u(m) \mathds{1}_{\mc C_m},
\end{equation*}
where the sum is to be understood in the sense of pointwise convergence. First of all, the map $J : \ell^\infty \tend L^\infty$ is bounded. Next, if $u \in c_0$, then $Ju$ has limit zero at $|x| \rightarrow + \infty$, and so $Ju \in C_0 \subset \mc S'_h \cap L^\infty$ (Example \ref{ex:c2C0} shows that the space $C_0$ of continuous function that tend to zero at infinity lies in $\mc S'_h$). The inclusion $J(c_0) \subset \ker(J)$ allows us to define a quotient map $J'$ such that the following diagram commutes:
\begin{equation*}
\xymatrix{
\ell^\infty \ar[rr]^J \ar@{->>}[d] & & L^\infty \ar@{->>}[d] \\
\ell^\infty / c_0 \ar[rr]^{J'} & & L^\infty / (\mc S'_h \cap L^\infty)
}
\end{equation*}
To fall back on the arguments of Theorem \ref{t:c2nonSepB}, we must now show the converse: that if $Ju \in \mc S'_h \cap L^\infty$ then we must have $u \in c_0$. For this, we prove an inequality that is quite analogous to \eqref{eq:c2IneqQuot}. We show that
\begin{equation}\label{eq:IneqQuotLinfty}
\forall f \in L^\infty, \qquad \limss_{\lambda \rightarrow + \infty } \left| \langle f, \psi_\lambda \rangle_{L^\infty \times L^1} \right| \leq {\rm dist} (f, \mc S'_h \cap L^\infty) := \inf_{h \in \mc S'_h \cap L^\infty} \| f - h \|_{L^\infty}.
\end{equation}
The argument is \textsl{mutatis mutandi} the same as for \eqref{eq:c2IneqQuot}. On the one hand, in light of Proposition \ref{p:c2SpHequiv}, it is clear that for any $h \in \mc S'_h \cap L^\infty$ we must have $\langle h, \psi_\lambda \rangle_{L^\infty \times L^1} \tend 0$. On the other hand, for $f \in L^\infty$ and $h \in \mc S'_h \cap L^\infty$, we have
\begin{equation*}
\left| \langle f, \psi_\lambda \rangle_{L^\infty \times L^1} \right| \leq \| f - h \|_{L^\infty} \| \psi \|_{L^1} + \left| \langle h, \psi_\lambda \rangle_{L^\infty \times L^1} \right|.
\end{equation*}
By taking the limit $\lambda \rightarrow + \infty$, we deduce \eqref{eq:IneqQuotLinfty}.

\medskip

With \eqref{eq:IneqQuotLinfty} at our disposal, we may show that $Ju \notin \mc S'_h \cap L^\infty$ if $u$ does not belong to $c_0$. We start by fixing a $u \in \ell^\infty$. Consider a $\lambda > 0$ whose value will be decided later, we have
\begin{equation}\label{eq:c2IntHuSum}
\int \psi_\lambda Ju = \sum_{m = 0}^\infty u(m) \int_{\mc C_m} \psi_\lambda = \sum_{m = 0}^\infty u(m) \int_{\lambda^{-1} \mc C_m} \psi.
\end{equation}
Let $\epsilon > 0$ and fix a $R > 0$ such that
\begin{equation*}
\| \mathds{1}_{|x| \geq R} \psi (x) \|_{L^1} = \| \mathds{1}_{| x| \geq R \lambda} \psi_\lambda (x) \|_{L^1} \leq \epsilon.
\end{equation*}
On the one hand, the terms of \eqref{eq:c2IntHuSum} at high indices $m$ will be small. More precisely,
\begin{equation}\label{eq:c2nonSepL1}
\left|\sum_{r_m \geq \lambda R} u(m) \int_{\lambda^{-1} \mc C_m} \psi  \right| \leq \| u \|_{\ell^\infty} \int_{|x| \geq R} |\psi (x)| {\rm d}x \leq \epsilon \| u \|_{\ell^\infty}.
\end{equation}
On the other hand, terms at low indices will also be small if $\lambda$ is large, because the annuli $\lambda^{-1} \mc C_m$ have a small measure: fix a $M \geq 0$, then
\begin{equation}\label{eq:c2nonSepL2}
\left| \sum_{m=0}^{M-1} u(m) \int_{\lambda^{-1} \mc C_m} \psi \right| \leq \| u \|_{\ell^\infty} \| \psi \|_{L^\infty} \frac{1}{\lambda^d} \sum_{m=0}^{M-1} |\mc C_m| = \| u \|_{\ell^\infty} \| \psi \|_{L^\infty} \left( \frac{r_M}{\lambda} \right)^d.
\end{equation}

\medskip

We set the values of the $r_m$, and $\lambda$ so that both sums \eqref{eq:c2nonSepL1} and \eqref{eq:c2nonSepL2} are negligible and only the contribution of one term matters. Fix a $M \geq 0$ and set
\begin{equation*}
r_m = 2^{m^2} \qquad \text{and} \qquad \lambda = r_{M+1}
\end{equation*}
so that $r_M / \lambda = O(4^{-M})$. The sum in \eqref{eq:c2nonSepL1} ranges on all indices $m$ such that $r_m \geq r_{M+1} R$, that is all $m \geq 0$ with
\begin{equation*}
2^{m^2} \geq 2^{2M+1} 2^{M^2} R.
\end{equation*}
Therefore, any $m \geq M+1$ is included in that sum if $M$ is taken large enough that $2^{2M+1} R > 1$. For such $M$, we may bound the difference between the full sum \eqref{eq:c2IntHuSum} and the $M$-th term:
\begin{equation*}
\left| \int \psi_\lambda Ju - u(M) \int_{\mc C_M} \psi_\lambda \right| \leq \| u \|_{\ell^\infty} \left( \epsilon + \frac{\| \psi \|_{L^\infty}}{4^M} \right).
\end{equation*}
Finally, the principal term is equal to $u(M)\int \psi$ up to a small remainder: by using exactly the same bound as in \eqref{eq:c2nonSepL1} and \eqref{eq:c2nonSepL2}, we see that
\begin{equation*}
\left| \int_{{}^c \mc C_M} \psi_\lambda \right| \leq \epsilon + \frac{\| \psi \|_{L^\infty}}{4^M}.
\end{equation*}
Therefore, by taking small $\epsilon$, large $M$ and $\lambda = r_{M+1}$, we see that we can make the bracket $\langle Ju, \psi_\lambda \rangle_{L^\infty \times L^1}$ arbitrarily close to any $u(M)$, so all the accumulation points of $u$ are also accumulation points of the bracket as $\lambda \rightarrow + \infty$. We deduce:
\begin{equation*}
\| u \|_{\ell^\infty / c_0} = \limss_{M \rightarrow + \infty} |u(M)| \leq \limss_{\lambda \rightarrow + \infty} \left| \langle Ju, \psi_\lambda \rangle_{L^\infty \times L^1} \right|,
\end{equation*}
which ends the proof.
\end{proof}

\subsection{Non-Complementation of the Intersection}

In this final Section, we exploit the embedding $J : \ell^\infty / c_0 \tend L^\infty / (\mc S'_h \cap L^\infty)$ we have constructed in Theorem \ref{t:c2nonSepL} above to show that $\mc S'_h \cap L^\infty$ is not complemented in $L^\infty$.

\begin{thm}\label{t:c2ComplL}
The space $\mc S'_h \cap L^\infty$ is not complemented in $L^\infty$. There exists no decomposition $L^\infty = (\mc S'_h \cap L^\infty) \oplus G$ with continuous projections.
\end{thm}

The proof of Theorem \ref{t:c2ComplL} will be very similar to that of Theorem \ref{t:c2ComplB}. In fact, we may give an abstract general principle which captures the essence of Whitley's argument and ``lifts'' it to another Banach space $X$ provided it contains a ``suitable'' copy of $\ell^\infty$. Theorem \ref{t:c2ComplL} is a direct consequence of Theorem \ref{t:c2nonSepL} and the following Proposition.

\begin{prop}\label{p:c1AbstractXE}
Let $X$ be a Banach space that has the following property: there exists a countable family $(g_n)_{n \geq 1}$ of bounded linear functionals $g_n \in X'$ such that
\begin{equation}\label{eq:countAssump}
\forall f \in X, \qquad  \Big( \forall n \geq 1, \; \langle g_n, f \rangle_{X' \times X} = 0 \Big) \; \Rightarrow \; f = 0.
\end{equation}
In particular, if $X$ has a separable predual, then this property is fulfilled. Now, consider a closed subspace $E \subset X$ and assume the existence of a bounded map $J : \ell^\infty \tend X$ that defines an embedding of the quotient $J': \ell^\infty / c_0 \tend X/E$ such that the diagram
\begin{equation*}
\xymatrix{
\ell^\infty \ar[rr]^J \ar@{->>}[d] & & X \ar@{->>}[d] \\
\ell^\infty / c_0 \ar[rr]^{J'} & & X/E
}
\end{equation*}
commutes (the vertical arrows are again the natural projections). Then $E$ is not complemented in $X$.
\end{prop}

\begin{proof}[Proof of Proposition \ref{p:c1AbstractXE}]
Consider $T : X \tend X$ a bounded operator such that $E \subset \ker(T)$ and let $(A_i)_{i \in I}$ be the family of subspaces given by Lemma \ref{l:c2countArg}. We show that there must be a $i \in I$ with $J \big( \ell^\infty (A_i) \big) \subset \ker(T)$. Assume on the contrary that for all $i \in I$ there is a $u_i$ such that $TJu_i \neq 0$. Then, if $(g_n)_n$ is as in the statement of Proposition \ref{p:c1AbstractXE},
\begin{equation*}
I = \{ i \in I, \; TJu_i \neq 0 \} = \bigcup_{k, n \geq 0} \left\{ i \in I, \quad \left| \langle g_n, TJu_i \rangle_{X' \times X} \right| \geq \frac{1}{k+1} \right\} := \bigcup_{k, n \geq 0} I_{k, n}.
\end{equation*}
Now, since $I$ is uncountable, there must be indices $k, n \geq 0$ such that $I_{k, n}$ is also uncountable. Next, for any finite $F \subset I_{k, n}$, let
\begin{equation*}
y = \sum_{i \in F} \alpha_i u_i, \qquad \text{where} \qquad \alpha_i = \frac{\langle g_n, TJu_i \rangle_{X' \times X}}{\left| \langle Tg_n, Ju_i \rangle_{X' \times X} \right|}.
\end{equation*}
In particular, 
\begin{equation}\label{eq:ComplL1}
\left| \langle g_n, TJy \rangle_{X' \times X} \right| \geq \frac{|F|}{k+1}.
\end{equation}
Next, thanks to the properties of the $A_i$ given by Lemma \ref{l:c2countArg}, we can decompose the union of the $A_i$ in $\cup_{i \in F} A \sqcup B$, where $A$ and $B$ are given exactly as in \eqref{eq:UnionAB}. By setting $y = a+b = \mathds{1}_A y + \mathds{1}_B y$, we have $TJy = TJa$ and
\begin{equation*}
\left| \langle TJy, g_n \rangle_{X' \times X} \right| \leq \| T \|,
\end{equation*}
which contradicts \eqref{eq:ComplL1}, since the set $F$ can be chosen as large as desired.
\end{proof}

Another consequence of the abstract principle of Proposition \ref{p:c1AbstractXE} is that we may prove Theorem \ref{t:c2ComplB} in the case of the Lebesgue exponent $p=1$. 

\begin{cor}\label{c:p1case}
Let $s \geq d$. Then $C^s_1$ is not complemented in $\dot{B}^s_{1, \infty}$: there is no decomposition $\dot{B}^s_{1, \infty} = C^s_1 \oplus G$ with continuous projections.
\end{cor}

\begin{proof}
We aim at applying Proposition \ref{p:c1AbstractXE}. By virtue of Theorem \ref{t:c2nonSepB}, we already have constructed a map $J : \ell^\infty \tend \dot{B}^s_{1, \infty}$ that satisfies the assumptions of Proposition \ref{p:c1AbstractXE}. It only remains to show that \eqref{eq:countAssump} is fulfilled for the Banach space $\dot{B}^s_{1, \infty}$.

\medskip

This is an easy consequence of the embedding properties of homogeneous spaces (see Proposition 2.20 p. 64 in \cite{BCD}). Since $\dot{B}^s_{1, \infty} \subset \dot{B}^{s-d}_{\infty, \infty}$ and because $\dot{B}^{s-d}_{\infty, \infty}$ has a separable predual $\dot{B}^{s-d}_{\infty, \infty} = (\dot{B}^{d-s}_{1, 1})'$ by Theorem \ref{t:c2BesovDuality}, we may fix a dense sequence $(g_n)$ in the unit ball of $\dot{B}^{d-s}_{1, 1}$ that satisfies
\begin{equation*}
\forall f \in \dot{B}^{s}_{1, \infty}, \qquad \Big( \forall n \geq 1, \quad \big\langle g_n, f \big\rangle_{\dot{B}^{s-d}_{\infty, \infty} \times \dot{B}^{d-s}_{1, 1}} = 0 \Big) \; \Rightarrow \; f=0.
\end{equation*}
As a consequence, we may apply Proposition \ref{p:c1AbstractXE} to show that $C^s_1$ is uncomplemented in $\dot{B}^s_{1, \infty}$.
\end{proof}

\begin{rmk}
Of course, the abstract principle of Proposition \ref{p:c1AbstractXE} is in reality a small part of the proof, the core of the argument is the construction of the map $J'$. Nevertheless, the same arguments may be used in a number of different frameworks. Let us give an example of a seemingly very different situation where this principle works.

Let $H$ be a real separable Hilbert space and $X := \mc L(H)$ the space of bounded linear operators on $H$. Define the subspace $E = \mc K (H)$ of compact operators. N.J. Kalton showed (see Theorem 6 in \cite{Kalton}) that $\mc K(H)$ is uncomplemented in $\mc L(H)$. The argument, presented in a simpler form\footnote{Kalton's result is much more general and expresses the complementation of the space $\mc K (X, Y)$ of compact operators between two Banach spaces in terms of whether $\mc L(X, Y)$ contains or not a copy of $\ell^\infty$. The proof in \cite{CQ} is adapted to the simpler framework of Hilbert spaces.} in \cite{CQ} (Theorem 6.1, pp. 351-354), although phrased differently, can be reformulated to fit in our framework. Define, for any $u \in \ell^\infty$, the operator $Ju : H \tend H$ by
\begin{equation*}
Ju(x) = \sum_{n = 1}^\infty u(n) \langle e_n, x \rangle_H \, e_n,
\end{equation*}
where $(e_n)_{n \geq 1}$ is a Hilbert basis of $H$. Then $Ju$ is compact if and only if $u \in c_0$, as it is in that case a limit of finite rank operators, so we get an embedding $J' : \ell^\infty / c_0 \tend \mc L(H) / \mc K (H)$. In addition, any $T \in \mc L(H)$ is equal to $T = 0$ if and only if
\begin{equation*}
\forall n, m \geq 1, \qquad g_{n,m}(T) := \langle e_n, T e_m \rangle_H = 0,
\end{equation*}
so that there is a countable number of bounded linear maps $g_{n, m} : \mc L(H) \tend \R$ that fulfill the assumptions of Proposition \ref{p:c1AbstractXE}. Our abstract principle (Proposition \ref{p:c1AbstractXE}) therefore applies to show that $\mc K(H)$ is uncomplemented in $\mc L(H)$.

\end{rmk}

\end{document}